\documentclass[12pt,a4paper]{amsart}
\usepackage{amsfonts,color}
\usepackage{amsthm}
\usepackage{amsmath}
\usepackage{amscd}
\usepackage[latin2]{inputenc}
\usepackage{t1enc}
\usepackage[mathscr]{eucal}
\usepackage{indentfirst}
\usepackage{graphicx}
\usepackage{graphics}
\usepackage{pict2e}
\usepackage{epic}
\usepackage{url}
\usepackage{epstopdf}
\usepackage{comment}
\usepackage{todonotes}

\numberwithin{equation}{section}
\usepackage[margin=2.6cm]{geometry}

\numberwithin{equation}{section}

\usepackage{amsmath}
\usepackage{graphicx}
\usepackage[colorlinks=true, allcolors=blue]{hyperref}
\usepackage{tikz}
\usetikzlibrary{arrows.meta}
\usepackage{float}
\usepackage{subcaption}
\usetikzlibrary{arrows, automata, calc}

\theoremstyle{plain}
\newtheorem{Th}{Theorem}[section]
\newtheorem{Lemma}[Th]{Lemma}
\newtheorem{Prop}[Th]{Proposition}
\newtheorem{Cor}[Th]{Corollary}

\theoremstyle{definition}
\newtheorem{Def}[Th]{Definition}

\newtheorem{Conj}[Th]{Conjecture}

\newtheorem{Rem}[Th]{Remark}

\newtheorem{Prob}[Th]{Problem}

\def\arb{{\rm arb}}
\def\allarb{{\rm allarb}}
\def\sp{{\rm sp}}
\DeclareMathOperator{\conv}{Conv}

\begin{document}

\title{Extremal number of arborescences}

\author[A. Bandekar]{Aditya Bandekar}
\address{University of Toronto, Department of Mathematics, 40 St. George Street Toronto}
\email{aditya.bandekar@mail.utoronto.ca}

\author[P. Csikv\'ari]{P\'{e}ter Csikv\'{a}ri}
\address{HUN-REN Alfr\'ed R\'enyi Institute of Mathematics, H-1053 Budapest Re\'altanoda utca 13-15 \and ELTE: E\"{o}tv\"{o}s Lor\'{a}nd University Mathematics Institute, Department of Computer
Science H-1117 Budapest, 
P\'{a}zm\'{a}ny P\'{e}ter s\'{e}t\'{a}ny 1/C}
\email{peter.csikvari@gmail.com}

\author[B. Mascuch]{Benjamin Mascuch}
\address{Tufts University, Department of Mathematics, 177 College Ave
Medford, MA 02155}
\email{benjamin.mascuch@tufts.edu}

\author[D. T\'ark\'anyi]{Damj\'an T\'ark\'anyi}
\address{ELTE E\"{o}tv\"{o}s Lor\'{a}nd University, H-1117 Budapest, P\'{a}zm\'{a}ny P\'{e}ter s\'{e}t\'{a}ny 1/C}
\email{damjantarkanyi@gmail.com}

\author[M. Telekes]{M\'arton Telekes}
\address{Budapest University of Technology and Economics, M\"{u}egyetem rkp. 3., H-1111 Budapest, Hungary}
\email{telekesmarton@gmail.com}

\author[L. T\'othm\'er\'esz]{Lilla T\'othm\'er\'esz}
\address{ELTE E\"{o}tv\"{o}s Lor\'{a}nd University, Mathematics Institute, Department of Operations Research, H-1117 Budapest, 
P\'{a}zm\'{a}ny P\'{e}ter s\'{e}t\'{a}ny 1/C}
\email{lilla.tothmeresz@ttk.elte.hu}

\thanks{
PC and LT were supported by the MTA-R\'enyi Counting in Sparse Graphs ''Momentum'' Research
Group. PC was also supported by the Dynasnet European Research Council Synergy project -- grant number ERC-2018-SYG 810115.
LT was also supported by the National Research, Development and Innovation Office of Hungary -- NKFIH, grant no.\ 132488, by the J\'anos Bolyai Research Scholarship of the Hungarian Academy of Sciences, and by the \'UNKP-23-5 New National Excellence Program of the Ministry for Innovation and Technology, Hungary.
}

\begin{abstract}
In this paper we study the following extremal graph theoretic problem: Given an undirected Eulerian graph $G$, which Eulerian orientation minimizes or maximizes the number of arborescences? 
We solve the minimization for the complete graph $K_n$, the complete bipartite graph $K_{n,m}$, and for the so-called double graphs, where there are even number of edges between any pair of vertices.

In fact, for $K_n$ we prove the following stronger statement. If $T$ is a tournament on $n$ vertices with out-degree sequence $d_1^+,\dots ,d^+_n$, then
$$\mathrm{allarb}(T)\geq \frac{1}{n}\left(\prod_{k=1}^n(d^+_k+1)+\prod_{k=1}^nd^+_k\right),$$
where $\allarb(T)$ is the total number of arborescences. Equality holds if and only if $T$ is a locally transitive tournament.

We also give an upper bound for the number of arborescences of an Eulerian orientation for an arbitrary graph $G$. This upper bound can be achieved on $K_n$ for infinitely many $n$. 
\end{abstract}

\maketitle

\section{Introduction}

An arborescence of a directed graph $D$ rooted at some vertex $v$ is a spanning tree of $D$ such that each vertex different from $v$ has out-degree $1$, and the root vertex $v$ has out-degree $0$. In other words, every edge of the spanning tree is oriented towards the root vertex. We denote the number of arborescences rooted at vertex $v$ by $\arb(D,v)$.
Given an Eulerian digraph $D$ the quantity $\arb(D,v)$ does not depend on $v$ (this follows from the BEST theorem \cite{de1951circuits}) and we will simply denote it by $\arb(D)$. For a not necessarily Eulerian digraph let 
$$\text{allarb}(D)=\sum_{v\in V}\arb(D,v).$$

In this paper we study the following extremal graph theoretical problems. 
\begin{Prob}\label{prob:minmax_Eulerian}
Given an undirected Eulerian graph $G$ which Eulerian orientation $O$ minimizes or maximizes $\arb(O)$?
\end{Prob}
\begin{Prob}\label{prob:minmax_general}
Among all orientations of an undirected graph, which orientation minimizes or maximizes the quantity $\allarb(O)$?
\end{Prob}

Problem \ref{prob:minmax_Eulerian} is motivated by a geometric problem, namely the study of the symmetric edge polytope of graphs and regular matroids. It turns out that Problem \ref{prob:minmax_Eulerian} is equivalent to finding the facet of the symmetric edge polytope of the cographic matroid of $G$ with minimal or maximal volume. For a more detailed explanation of this connection, see Section \ref{ss:geom_motivation}. Note that we will not use this connection in this paper, and every result in this paper can be understood without understanding the geometric motivation.

Problem~\ref{prob:minmax_Eulerian} is also closely connected to the celebrated BEST theorem due to de Bruijn, van Aardenne--Ehrenfest \cite{de1951circuits}, Smith and Tutte claiming that the number of Eulerian tours of an Eulerian digraph $D$ is
$$\arb(D)\prod_{v\in V}(d^+(v)-1)!,$$
where $d^+(v)$ is the out-degree of vertex $v$. This theorem shows that the maximizing or minimizing Eulerian orientation also maximizes or minimizes the number of Eulerian tours.

\subsection{Results.} Our first result is a general lower bound for $\allarb$ on tournaments, that is, arbitrary orientations of complete graphs. To spell out the case of equality of this theorem we need the following definition.

\begin{Def} A digraph is called locally transitive if for every vertex $v$ the out-neighbors $N^+(v)$ and the in-neighbors $N^-(v)$ both induce a transitive tournament.
\end{Def}

\begin{Th} \label{lower_bound_K_n}
Let $T$ be a tournament on $n$ vertices with out-degree sequence $d_1^+,\dots ,d^+_n$. Then
$$\mathrm{allarb}(T)\geq \frac{1}{n}\left(\prod_{k=1}^n(d^+_k+1)+\prod_{k=1}^nd^+_k\right).$$
Equality holds if and only if $T$ is a locally transitive tournament.
\end{Th}

Note that Theorem~\ref{lower_bound_K_n} is not true for general digraphs as it can occur that a digraph has no arborescence at all. 

\begin{Def}
    For odd $n=2d+1$, we call the following tournament the \emph{swirl tournament on $n$ vertices} and we will denote it by $SW_n$: Let the vertex set be $[n]$, and let the edges be $\{(i,i+k \ \bmod \ n) : i\in [n], 1\leq k\leq d\}$. See Figure \ref{fig:swirl_tournament} for an example.
\end{Def}

The following theorem is a simple specialization of Theorem~\ref{lower_bound_K_n} for Eulerian tournaments.

\begin{Th} \label{Eulerian_lower_bound_K_n}
Let $T_n$ be an Eulerian tournament on $n$ vertices. Then
\begin{align*}
\arb(T_n) \geq \frac{1}{n^2}\left(\left(\frac{n+1}{2}\right)^{n} + \left(\frac{n-1}{2}\right)^{n}\right)
\end{align*}
with equality if and only if $T_n\cong SW_n$. 
\end{Th}

\begin{figure}[H]
    \begin{tikzpicture}[nodes=state]
    \tikzstyle{every node}=[draw,shape=circle,inner sep=0,minimum size=2mm,fill=black]
    \tikzset{edge/.style = {->,> = latex}}
    \def \number {7}
    \def \radius {2cm}
    \def \degree {360/\number}
    \foreach \s in {1,...,\number}
    {
        \node at ({\degree * (\s -1)}:\radius) (\s) {};
    }
    \draw[-{Latex[length=5pt]}] (1) to (2);
    \draw[-{Latex[length=5pt]}] (1) to (3);
    \draw[-{Latex[length=5pt]}] (1) to (4);
    \draw[-{Latex[length=5pt]}] (2) to (3);
    \draw[-{Latex[length=5pt]}] (2) to (4);
    \draw[-{Latex[length=5pt]}] (2) to (5);
    \draw[-{Latex[length=5pt]}] (3) to (4);
    \draw[-{Latex[length=5pt]}] (3) to (5);
    \draw[-{Latex[length=5pt]}] (3) to (6);
    \draw[-{Latex[length=5pt]}] (4) to (5);
    \draw[-{Latex[length=5pt]}] (4) to (6);
    \draw[-{Latex[length=5pt]}] (4) to (7);
    \draw[-{Latex[length=5pt]}] (5) to (6);
    \draw[-{Latex[length=5pt]}] (5) to (7);
    \draw[-{Latex[length=5pt]}] (5) to (1);
    \draw[-{Latex[length=5pt]}] (6) to (7);
    \draw[-{Latex[length=5pt]}] (6) to (1);
    \draw[-{Latex[length=5pt]}] (6) to (2);
    \draw[-{Latex[length=5pt]}] (7) to (1);
    \draw[-{Latex[length=5pt]}] (7) to (2);
    \draw[-{Latex[length=5pt]}] (7) to (3);
    \end{tikzpicture}
    \caption{The swirl tournament $SW_7$, which is the minimizing Eulerian orientation of $K_7$.} \label{fig:swirl_tournament}
\end{figure}
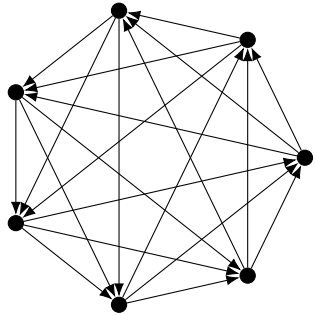

The following theorem is another immediate corollary of Theorem~\ref{lower_bound_K_n}.

\begin{Th} \label{transitive_tournament}
Let $T_n$ be a tournament on $n$ vertices, and let $TR_n$ be the transitive tournament on $n$ vertices. Then
$$\mathrm{allarb}(T_n)\geq \mathrm{allarb}(TR_n)$$
with equality if and only if $T_n\cong TR_n$. 
\end{Th}

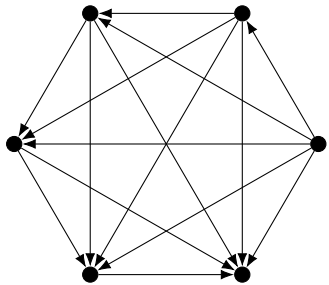
\begin{figure}[H]
    \begin{tikzpicture}[nodes=state]
    \tikzstyle{every node}=[draw,shape=circle,inner sep=0,minimum size=2mm,fill=black]
    \tikzset{edge/.style = {->,> = latex}}
    \def \number {6}
    \def \radius {2cm}
    \def \degree {360/\number}
    \foreach \s in {1,...,\number}
    {
        \node at ({\degree * (\s -1)}:\radius) (\s) {};
    }
    \foreach \s in {2,...,\number}
    {
        \pgfmathsetmacro\result{\s - 1}
        \foreach \i in {1,...,\result}
        {
            \draw[-{Latex[length=5pt]}] (\i) to (\s);
        }
    }
    \end{tikzpicture}
    \caption{The transitive tournament $TR_6$}
\end{figure}

Theorem \ref{lower_bound_K_n} is surprisingly tight as the following simple statement is true for any digraph.

\begin{Prop} \label{trivial_upper_bound}
Let $D$ be a digraph on $n$ vertices with out-degree sequence $d_1^+,\dots ,d^+_n$. Then
$$\mathrm{allarb}(D)\leq \sum_{k=1}^n\prod_{j\neq k}d_j^+< \prod_{k=1}^n(d^+_k+1).$$
\end{Prop}

Next, we give an upper bound for $\arb(O)$ for Eulerian orientations of $K_n$ (for $n$ odd) that is tight for infinitely many values of $n$.
Let us start with a general upper bound for $\allarb(O)$ for any graph $G$, which is also tight in several cases.

\begin{Th}\label{conj:upper_bound} 
Let $D$ be simple directed graph on $n$ vertices and $m$ edges, with out-degree sequence $d^+_1,\ldots,d^+_n$.
Then
\begin{align*}
\allarb(D) \leq \left(\frac{1}{n-1}\right)^{\frac{n-1}{2}}\left(\sum_{i=1}^n (d^+_i)^2 + m\right)^{\frac{n-1}{2}}.
\end{align*}
In particular, if $G$ is a simple Eulerian graph with degree sequence $d_1,\dots ,d_n$, and $O$ is an Eulerian orientation of $G$, then we have
\begin{align*}
\arb(O) \leq \frac{1}{n}\left(\frac{1}{n-1}\right)^{\frac{n-1}{2}}\left(\frac{1}{4}\sum_{i=1}^n d_i^2 + m\right)^{\frac{n-1}{2}}.
\end{align*}
\end{Th}

Next we study tournaments on $n$ vertices, where $n$ is an odd integer. It turns out that for certain $n$ the so-called Hadamard tournaments will be the maximizing orientations. 

\begin{Def}
A tournament on $n$ vertices is an Hadamard tournament if its adjacency matrix $A$ satisfies $AA^T = \frac{n+1}{4}I+\frac{n-3}{4}J$ where $J$ is the $n\times n$ matrix with each entry being $1$. Note that Hadamard tournaments are sometimes referred to as doubly regular tournaments or homogeneous tournaments in the literature.
\end{Def}

\begin{figure}[H]
    \begin{tikzpicture}[nodes=state]
    \tikzstyle{every node}=[draw,shape=circle,inner sep=0,minimum size=2mm,fill=black]
    \tikzset{edge/.style = {->,> = latex}}
    \def \number {7}
    \def \radius {2cm}
    \def \degree {360/\number}
    \foreach \s in {1,...,\number}
    {
        \node at ({\degree * (\s -1)}:\radius) (\s) {};
    }
    \draw[-{Latex[length=5pt]}] (1) to (2);
    \draw[-{Latex[length=5pt]}] (1) to (3);
    \draw[-{Latex[length=5pt]}] (1) to (5);
    \draw[-{Latex[length=5pt]}] (2) to (3);
    \draw[-{Latex[length=5pt]}] (2) to (4);
    \draw[-{Latex[length=5pt]}] (2) to (6);
    \draw[-{Latex[length=5pt]}] (3) to (4);
    \draw[-{Latex[length=5pt]}] (3) to (5);
    \draw[-{Latex[length=5pt]}] (3) to (7);
    \draw[-{Latex[length=5pt]}] (5) to (2);
    \draw[-{Latex[length=5pt]}] (5) to (6);
    \draw[-{Latex[length=5pt]}] (5) to (7);
    \draw[-{Latex[length=5pt]}] (4) to (1);
    \draw[-{Latex[length=5pt]}] (4) to (5);
    \draw[-{Latex[length=5pt]}] (4) to (6);
    \draw[-{Latex[length=5pt]}] (6) to (1);
    \draw[-{Latex[length=5pt]}] (6) to (3);
    \draw[-{Latex[length=5pt]}] (6) to (7);
    \draw[-{Latex[length=5pt]}] (7) to (1);
    \draw[-{Latex[length=5pt]}] (7) to (2);
    \draw[-{Latex[length=5pt]}] (7) to (4);
    \end{tikzpicture}
    \caption{A Hadamard tournament, namely the Paley tournament of order $7$.}
\end{figure}
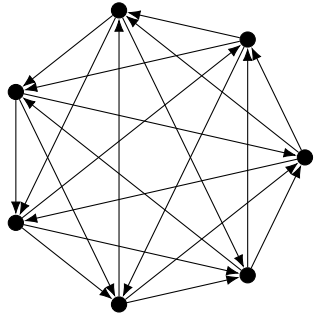

Recall that a $\pm 1$ matrix $H$ of size $N\times N$ is Hadamard if $H^TH=NI_N$. There is a simple connection between Hadamard matrices and tournaments. There is an Hadamard tournament on $n$ vertices if and only if there is a skew Hadamard matrix of size $n+1$ (see \cite{reid1972doubly}). It is known that the size of an Hadamard matrix is either $1,2$ or divisible by $4$ (see Theorem 18.1 of \cite{van2001course}), this means that for $n>1$ we can only expect an Hadamard tournament on $n$ vertices if $n\equiv 3\ (\bmod\ 4)$. 

Specializing Theorem~\ref{conj:upper_bound} to the complete graph $K_n$ we get an upper bound for the number of arborescences of Eulerian tournaments.
In this case we can also characterize the equality case.

\begin{Th}\label{upper_bound_for_K_n}
Let $T_n$ be an Eulerian tournament on $n$ vertices. Then
\begin{align*}
\arb(T_n)\leq \frac{1}{n}\left(\frac{n(n+1)}{4}\right)^{\frac{n-1}{2}}
\end{align*}
with equality if and only if $T_n$ is a Hadamard tournament.
\end{Th}

Next, let us study the minimization problem for the complete bipartite graph $K_{n,m}$. It turns out that in this case the minimization problem for $\mathrm{allarb}(O)$ among all orientations is trivial. Indeed, if a graph $G$ has two non-adjacent vertices, then the minimization problem for $\mathrm{allarb}(O)$ is trivial as orienting each incident edge toward these vertices will immediately imply that there are no arborescences in the obtained directed graph as there can be only one root, and any non-root vertex should have at least one out-going edge.
So among simple graphs this question is only non-trivial if there are no non-adjacent vertices, that is, when $G$ is a complete graph. In particular, for $K_{n,m}$ the minimal number of arborescences among all orientations is simply $0$. This means that only minimization among Eulerian orientations is worth considering.

\begin{Th} \label{lower_bound_K_nm}
Let $O$ be an Eulerian orientation of $K_{n,m}$, where $n$ and $m$ are an even integers. Then
\begin{align*}
\arb(O)\geq \left(\frac{m}{2}\right)^{n-1}\left(\frac{n}{2}\right)^{m-1}.
\end{align*}
The following orientation of $K_{n,m}$ achieves the lower bound and is the unique minimizer up to isomorphism: take an oriented $4$-cycle and blow up every second vertex with $\frac{n}{2}$ vertices and every second vertex with $\frac{m}{2}$ vertices.
\end{Th}

\begin{figure}[H]
\begin{tikzpicture}[nodes=state]
    \tikzstyle{every node}=[draw,shape=circle,inner sep=0,minimum size=2mm,fill=black]
    \tikzset{edge/.style = {->,> = latex}}
    \node at (0,1) (1) {};
    \node at (0,2) (2) {};
    \node at (0,3) (3) {};
    \node at (4,1) (4) {};
    \node at (4,2) (5) {};
    \node at (4,3) (6) {};
    \node at (1.5,0) (7) {};
    \node at (2.5,0) (8) {};
    \node at (1.5,4) (9) {};
    \node at (2.5,4) (10) {};
    \draw[-{Latex[length=5pt]}] (1) to (9);
    \draw[-{Latex[length=5pt]}] (1) to (10);
    \draw[-{Latex[length=5pt]}] (2) to (9);
    \draw[-{Latex[length=5pt]}] (2) to (10);
    \draw[-{Latex[length=5pt]}] (3) to (9);
    \draw[-{Latex[length=5pt]}] (3) to (10);
    \draw[-{Latex[length=5pt]}] (9) to (4);
    \draw[-{Latex[length=5pt]}] (10) to (4);
    \draw[-{Latex[length=5pt]}] (9) to (5);
    \draw[-{Latex[length=5pt]}] (10) to (5);
    \draw[-{Latex[length=5pt]}] (9) to (6);
    \draw[-{Latex[length=5pt]}] (10) to (6);
    \draw[-{Latex[length=5pt]}] (4) to (7);
    \draw[-{Latex[length=5pt]}] (4) to (8);
    \draw[-{Latex[length=5pt]}] (5) to (7);
    \draw[-{Latex[length=5pt]}] (5) to (8);
    \draw[-{Latex[length=5pt]}] (6) to (7);
    \draw[-{Latex[length=5pt]}] (6) to (8);
    \draw[-{Latex[length=5pt]}] (7) to (1);
    \draw[-{Latex[length=5pt]}] (7) to (2);
    \draw[-{Latex[length=5pt]}] (7) to (3);
    \draw[-{Latex[length=5pt]}] (8) to (1);
    \draw[-{Latex[length=5pt]}] (8) to (2);
    \draw[-{Latex[length=5pt]}] (8) to (3);

    \end{tikzpicture}
    \caption{The minimizing Eulerian orientation of $K_{4,6}$}
\end{figure}
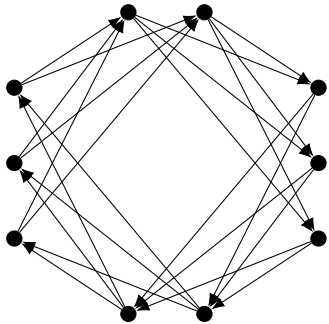

Our last result is the solution of the Eulerian minimization problem for the so-called double graphs. A graph $G$ is a double graph if there is an even number of edges between any two vertices. In this case we show that the following Eulerian orientation minimizes $\arb(O)$: for each $u,v\in V$, half the edges between $u$ and $v$ are oriented toward $u$, and half of them toward $v$. We call this orientation the symmetric orientation of $G$.

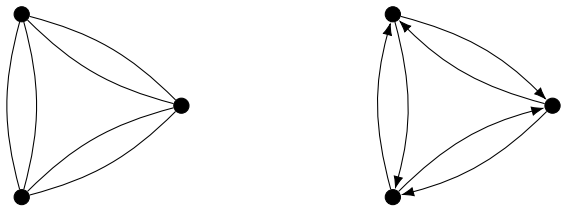
\begin{figure}[H]
    \centering
    \begin{subfigure}[t]{0.3\textwidth}
        \centering
        \begin{tikzpicture}[nodes=state,scale=.7]
        \tikzstyle{every node}=[draw,shape=circle,inner sep=0,minimum size=2mm,fill=black]
        \tikzset{edge/.style = {->,> = latex}}
        \def \number {3}
        \def \radius {2cm}
        \def \degree {360/\number}
        \foreach \s in {1,...,\number}
        {
            \node at ({\degree * (\s -1)}:\radius) (\s) {};
        }
        \foreach \s in {2,...,\number}
        {
            \pgfmathsetmacro\result{\s - 1}
            \foreach \i in {1,...,\result}
            {
                \draw[bend left=15] (\i) to (\s);
                \draw[bend left=15] (\s) to (\i);
            }
        }
        \end{tikzpicture}
    \end{subfigure}
    ~ 
    \begin{subfigure}[t]{0.3\textwidth}
        \centering
        \begin{tikzpicture}[nodes=state,scale=.7]
        \tikzstyle{every node}=[draw,shape=circle,inner sep=0,minimum size=2mm,fill=black]
        \tikzset{edge/.style = {->,> = latex}}
        \def \number {3}
        \def \radius {2cm}
        \def \degree {360/\number}
        \foreach \s in {1,...,\number}
        {
            \node at ({\degree * (\s -1)}:\radius) (\s) {};
        }
        \foreach \s in {2,...,\number}
        {
            \pgfmathsetmacro\result{\s - 1}
            \foreach \i in {1,...,\result}
            {
                \draw[-{Latex[length=5pt]}, bend left=15] (\i) to (\s);
                \draw[-{Latex[length=5pt]}, bend left=15] (\s) to (\i);
            }
        }
        \end{tikzpicture}
    \end{subfigure}
    \caption{A double graph on the left and its symmetric orientation on the right.}
    \label{fig:double_graph}
\end{figure}

\begin{Th}\label{thm:double_graph_minimizer}
Let $G$ be a connected double graph. Then the unique Eulerian orientation $O$ of $G$ minimizing $\arb(O)$ is the symmetric orientation.
\end{Th}

\subsection{A geometric motivation}
\label{ss:geom_motivation}

One of the motivations for studying the number of arborescences in Eulerian orientations comes from geometry, more precisely, from studying the volumes of facets of the so-called symmetric edge polytope of a bipartite graph.

The symmetric edge polytope of a graph $G=(V,E)$ was defined by Matsui, Higashitani, Nagazawa, Ohsugi, and Hibi \cite{Sym_edge_appearance} as the following polytope:
\[\mathcal{Q}_G=\conv(\{\, \mathbf1_u-\mathbf1_v, \, \mathbf1_v-\mathbf1_u \mid uv\in E\,\})\subset\mathbb R^V\]
Here $\mathbf{1}_v$ denotes the vector where the coordinate corresponding to $v$ is 1, and the rest of the coordinates are 0.
This polytope has recently garnered considerable interest \cite{arithm_symedgepoly,DDM19,Sym_edge_appearance,OT,ChenDavis} due to its nice combinatorial properties. Moreover, it has connections to the Kuramoto synchronization model of physics, where its volume yields an upper bound for the number of steady states \cite{ChenDavis}. Recently, the symmetric edge polytope was also generalized to regular matroids \cite{Eulerian_greedoid,DAli}.

Let us call a matroid bipartite if each circuit has even cardinality.
It turns out (see \cite{deg_interior_poly}), that for a bipartite cographic matroid, the facets of the symmetric edge polytope correspond to Eulerian orientations of the dual (Eulerian) graph, and the volumes of these facets are the arborescence numbers of the corresponding Eulerian orientations. 

Hence finding Eulerian orientations of an Eulerian graph with minimal or maximal arborescence number corresponds to finding facets of minimal or maximal volume for the symmetric edge polytope of the cographic matroid.

\bigskip
\noindent \textbf{This paper is organized as follows.} 
In the next section we introduce the necessary tools to study the number of arborescences. In Section~\ref{sect:lower_bounds} we prove the lower bound results, that is, Theorems~\ref{lower_bound_K_n}, \ref{Eulerian_lower_bound_K_n}, \ref{lower_bound_K_nm} and \ref{thm:double_graph_minimizer}. In Section~\ref{sect:upper_bounds} we prove the upper bound results, Theorem~\ref{upper_bound_for_K_n} and Proposition~\ref{trivial_upper_bound}.
We end the paper with some concluding remarks and open problems.

\section{Preliminaries} \label{Preliminaries}

\noindent \textbf{Notation.} Throughout the paper $G=(V,E)$ denotes a graph, $D$ denotes a directed graph, an $O$ is an orientation of some undirected graph $G$. $K_n$ denotes the complete graph on $n$ vertices. $K_{n,m}$ denotes the complete bipartite graph with parts of size $n$ and $m$. The notation $[n]$ stands for $\{1,2,\dots, n\}$.

The matrix $I_k$ is the $k\times k$ identity matrix, the matrix $J_k$ is $k\times k$ matrix consisting only of $1$'s. If the size of the matrix is clear from the context, we drop the subscript. 
We denote by $\underline{0}$ the all-zero vector, and by $\underline{1}$ the all-one vector (we do not indicate the sizes as it will be clear from the context). 

For a matrix $M$ and $S,T\subseteq [n]$ the $M_{S,T}$ is the submatrix of $M$ with rows from $S$ and columns from $T$. We use $M_{\overline{i},\overline{j}}$ for $M_{[n]\setminus \{i\},[n]\setminus \{j\}}$, that is, for the matrix obtained from $M$ by deleting the $i^{th}$ row and $j^{th}$ column.
We denote the $i^{th}$ column of $M$ by $M_{*,i}$, and the $i^{th}$ row by $M_{i,*}$.

We denote the characteristic polynomial of matrix $M$ by $\varphi_M$. That is, $\varphi_M(x)=\det(xI-M)$. Recall that $\varphi_M(x)=\prod_{i=1}^n (x-\lambda_i)$ where $\lambda_1, \dots ,\lambda_n$ are the eigenvalues of $M$ with multiplicity.

The Frobenius norm of a matrix $M$ is denoted by $||M||_F$. Recall that $||M||_F=\sqrt{\sum_{i,j}M_{i,j}^2}$.
\bigskip

For a directed graph $D$, we denote by $d^+(v)$ the out-degree of vertex $v$, by $d^-(v)$ the in-degree of vertex $v$. We denote by $m(u,v)$ the number of directed edges pointing from $u$ to $v$. Similarly, for an undirected graph $G$ the degree of a vertex $v$ is denoted by $d(v)$ and $m(u,v)$ denotes the number of edges between $u$ and $v$.

Given a directed graph $D$ we can associate matrices to $D$ in many different ways. We will need the following three matrices: the adjacency matrix, the skew-symmetric adjacency matrix and the Laplacian matrix.

\begin{Def}
For a directed graph $D$, let $A(D)$ refer to the adjacency matrix of $D$, where $A(D)_{ij}=m(v_i,v_j)$,
the number of directed edges from $v_i$ to $v_j$.
\end{Def}

\begin{Def}
For a directed graph $D$, let $L(D)$ refer to the Laplacian matrix of $D$, defined as
\[
    L(D)_{ij} = \begin{cases} 
        d^+(v_i) & \text{if } i = j, \\
        -m(v_i,v_j) & \text{if } i \neq j.
    \end{cases}
\]
If $G$ is an undirected graph, then the Laplacian matrix $L(G)$ is defined by 
\[
    L(G)_{ij} = \begin{cases} 
        d(v_i) & \text{if } i = j, \\
        -m(v_i,v_j) & \text{if } i \neq j.
    \end{cases}
\]
\end{Def}

The third matrix is the skew-symmetric adjacency matrix, that is slightly special in that we only associate it to orientations of simple graphs.
\begin{Def}
For a directed graph $D$ with a simple underlying graph $G$, let $M(D)$ refer to the skew-symmetric adjacency matrix of $D$, where $M(D)_{ij}=1$ if there is a directed edge from $v_i$ to $v_j$, $-1$ if there is a directed edge from $v_j$ to $v_i$, and $0$ otherwise. In other words, $M(D)=A(D)-A(D)^T$. \end{Def}

The following result of Tutte is fundamental for us.

\begin{Th}[Tutte's Matrix-Tree theorem \cite{tutte1948dissection}]
The number of arborescences of a digraph $D$ rooted at vertex $v_k$ is equal to $\det(L(D)_{\overline{k},\overline{k}})$.
\end{Th}

\noindent This is the counterpart of Kirchhoff's classical matrix tree theorem to directed graphs.

\begin{Th}[Kirchhoff's Matrix-Tree theorem]
The number of spanning trees of an undirected graph $G$ is equal to $\det(L(G)_{\overline{n},\overline{n}})$.
\end{Th}

$\det(L(D)_{\overline{k},\overline{k}})$ can be written as the product of the eigenvalues of $L(D)_{\overline{k},\overline{k}}$, but it turns out that it is more convenient to work with the eigenvalues of $L(D)$ instead. Lemma \ref{lem:product_of_nonzero_eigenvalues} below shows that the eigenvalues of $L(D)$ also give us meaningful input for studying the number of arborescences.

Concerning the eigenvalues of $L(D)$ one needs to be a bit careful. Unlike in the case of undirected graphs (where $L(G)$ is a symmetric positive semidefinite matrix) for digraphs it is not true anymore that the eigenvalues are real or that the matrix is positive semidefinite. It might even occur that $L(D)$ is not diagonalizable, that is, there is no basis consisting of eigenvectors of $L(D)$. Nevertheless, it is true that $L(D)\underline{1}=\underline{0}$ and we will refer to the corresponding $0$ eigenvalue as $\lambda_n$. Among the eigenvalues $\lambda_1,\dots ,\lambda_{n-1}$ there can be another $0$, but as the next lemma shows it can only occur if there is no arborescence in the digraph $D$.

\begin{Lemma}\label{lem:product_of_nonzero_eigenvalues}
For any digraph $D$ on $n$ vertices we have
\begin{equation}
\allarb(D)=\prod_{i=1}^{n-1} \lambda_{i} 
\end{equation}
where $\lambda_{1},\dots , \lambda_{n-1}$ are the eigenvalues of the Laplacian matrix different from $\lambda_n=0$.
\end{Lemma}

\begin{proof}
Take
\begin{align*}
\varphi_{L(D)}(x)= \prod\limits_{i=1}^{n} (x-\lambda_{i}) = x^{n} - a_{n-1} x^{n-1} + a_{n-2} x^{n-2} - \cdots +(-1)^{n-1} a_{1}x.
\end{align*}
Then, according to Vi\'ete's formula
\begin{align*}
a_{1} = \lambda_{2} \lambda_{3} \cdots \lambda_{n} + \lambda_{1} \lambda_{3} \cdots \lambda_{n} + \cdots + \lambda_{1} \lambda_{2} \cdots \lambda_{n-1}.
\end{align*} 
As $\lambda_{n} = 0$ we get $a_{1} = \prod_{i=1}^{n-1} \lambda_{i}$. Furthermore, by taking the cofactor expansion of $\varphi_{L(D)}(x)=\det(xI -L(G))$, we obtain the coefficient $a_{1} = \sum\limits_{k=1}^{n} \det(L(D)_{\overline{k},\overline{k}}) = \sum\limits_{r \in V(D)} \text{arb} (D,r)$ due to Tutte's Matrix-Tree theorem. Thus, $\prod\limits_{k=1}^{n-1}\lambda_{k} = \sum\limits_{r \in V(D)} \text{arb} (D,r)$.
\end{proof}

The following corollary of Lemma~\ref{lem:product_of_nonzero_eigenvalues} will be a key tool for us.

\begin{Cor}\label{cor:product_of_nonzero_eigenvalues_eulerian}
Let $D$ be an Eulerian digraph on $n$ vertices. Since $D$ is Eulerian, we have $\arb(D, u) = \arb(D,v)$ for all $u,v\in V(D)$. Thus
$$\arb(D)=\frac{1}{n}\prod_{i=1}^{n-1} \lambda_i.$$
\end{Cor}

\begin{Lemma} \label{J-perturbation}
Let $D$ be a digraph. Then
$$\det(L(D)+\alpha J_n)=n\alpha  \cdot \allarb(D).$$
\end{Lemma}

\begin{proof}
For the all-$1$ vector $\underline{1}$ we have $L(D)\underline{1}=\underline{0}$ and $(L(D)+\alpha J_n)\underline{1}=n\alpha \underline{1}$. Let $V=\langle \underline{1}\rangle^{\perp}$. For a vector $\underline{v}\in V$ we have $L(D)\underline{v}=(L(D)+\alpha J_n)\underline{v}$ showing that the remaining eigenvalues of $L(D)$ and $L(D)+\alpha J_n$ are the same. By writing up the two linear maps in a  basis of $V$ together with $\underline{1}$ they take the form of a block matrix 
$\left(\begin{array}{cc} c & \underline{u} \\
\underline{0} & M
\end{array}\right)$, where $c=0$ for $L(D)$ and $n\alpha$ for $L(D)+\alpha J_n$, respectively. So the rest of the eigenvalues are the eigenvalues of $M$ for both matrices. If $R$ denotes the multiset of these eigenvalues, then by Lemma~\ref{lem:product_of_nonzero_eigenvalues} we have
$$n\alpha\cdot \allarb(D)=n\alpha \prod_{\lambda \in R}\lambda=\det(L(D)+\alpha J_n).$$
\end{proof}

The following lemma is well-known, see Theorem 2.5.3 and its proof in \cite{horn2012matrix}. 

\begin{Lemma} \label{Frobenius}
Let $A\in \mathbb{C}^{n\times n}$ with eigenvalues $\lambda_1,\dots ,\lambda_n$. Then
$$\sum_{i=1}^n|\lambda_i|^2\leq \sum_{i,j}|A_{i,j}|^2$$
with equality if and only if $A$ is normal, that is, $A^*A=AA^*$.
\end{Lemma}

Finally, we will use the following basic facts about skew-symmetric matrices.

\begin{Lemma}[\cite{cayley1849determinants,ledermann1993note}]\label{skew-symmetric}
Let $M\in \mathbb{C}^{n\times n}$ be a skew-symmetric matrix, that is, $M^*=-M$. \\
(i) Then $M$ is normal, thus, $M$ has an orthonormal basis of eigenvectors. \\
(ii) The eigenvalues of $M$ are of the form $\alpha_1i,-\alpha_1i,\alpha_2i,-\alpha_2i,\dots$, where $\alpha_1,\alpha_2,\dots\in \mathbb{R}$. \\
(iii) $0$ is an eigenvalue of $M$ if $n$ is odd, and so $\det(M)=0$ in this case.\\
(iv) The determinant of $M$ is non-negative. In fact, the determinant is the square of the Pfaffian, and so $\det(M)$ is a perfect square if $M$ has only integer values.
\end{Lemma}

\section{Lower bounds} \label{sect:lower_bounds}

In this section we prove Theorems~\ref{lower_bound_K_n}, \ref{Eulerian_lower_bound_K_n}, \ref{transitive_tournament} and \ref{lower_bound_K_nm}.

\subsection{Lower bounds for tournaments}

\begin{proof}[Proof of Theorem~\ref{lower_bound_K_n}]
By Lemma~\ref{J-perturbation} we know that
$$\det(L(T)+\alpha J_n)=n\alpha \cdot \allarb(T).$$
Let us consider the matrix $2L(T)+J_n$. First observe that
$$\det(2L(T)+J_n)=2^n\det\left(L(T)+\frac{1}{2}J_n\right)=2^{n-1}n\cdot \allarb(T).$$
Furthermore, $2L(T)+J_n=D(T)-M(T)$, where $D(T)$ is the diagonal matrix consisting of the elements $2d^+_i+1$, and $M(T)$ is the skew-symmetric adjacency matrix of $T$. Hence
$$\det(2L(T)+J_n)=\det(D(T)-M(T))=\sum_{S\subseteq [n]}\left(\prod_{k\in S}(2d^+_k+1)\right)(-1)^{n-|S|}\det(M(T)_{S^c,S^c}),$$
where $S^c=[n]\setminus S$.
Here $M(T)_{S^c,S^c}$ is a skew-symmetric matrix. If it has odd size, then its determinant is $0$ according to Lemma~\ref{skew-symmetric}(iii). If it has even size, then its determinant is non-negative according to Lemma~\ref{skew-symmetric}(iv). Furthermore, it is an integer since all entries are integer. If $|S^c|=2k$, then the parity of the determinant of $M(T)_{S^c,S^c}$ is the same as the parity of the determinant of $J_{2k}-I_{2k}$ since the two matrices have the same parity entrywise and the determinant is a polynomial of the entries with integer coefficients. The eigenvalues of the latter matrix are $2k-1$ with multiplicity 1 with eigenvector $\underline{1}$, and $-1$ with multiplicity $2k-1$ with the eigenspace being all vectors orthogonal to $\underline{1}$, so $\det(J_{2k}-I_{2k})=(2k-1)(-1)^{2k-1}=-(2k-1)$ is odd. This means that $\det(M(T)_{S^c,S^c})$ is a non-negative odd number, that is, $\det(M(T)_{S^c,S^c})\geq 1$. Hence
$$\sum_{S\subseteq [n]}\left(\prod_{k\in S}(2d^+_k+1)\right)(-1)^{n-|S|}\det(M(T)_{S^c,S^c})\geq \sum_{S\subseteq [n] \atop |S|\equiv n\ (2)}\prod_{k\in S}(2d^+_k+1).$$
Observe that
\begin{align*}
\sum_{S\subseteq [n] \atop |S|\equiv n\ (2)}\prod_{k\in S}(2d^+_k+1)&=\frac{1}{2}\left(\prod_{k=1}^n((2d^+_k+1)+1)+\prod_{k=1}^n((2d^+_k+1)-1)\right)\\
&=2^{n-1}\left(\prod_{k=1}^n(d^+_k+1)+\prod_{k=1}^nd^+_k\right).
\end{align*}
Putting all these together, we get that 
$$2^{n-1}n\cdot \allarb(T)=\det(2L(T)+J)\geq 2^{n-1}\left(\prod_{k=1}^n(d^+_k+1)+\prod_{k=1}^nd^+_k\right).$$
Hence 
$$\allarb(T)\geq \frac{1}{n}\left(\prod_{k=1}^n(d^+_k+1)+\prod_{k=1}^nd^+_k\right).$$
This proves the inequality part of the theorem.

To have equality, we need that $\det(M(T)_{S^c,S^c})=1$ for each $S\subseteq [n]$ such that $|S^c|$ is even. As Lemma \ref{locally-transitive} shows below, this is equivalent to the tournament being locally transitive.
\end{proof}

\begin{Lemma} \label{locally-transitive}
The following are equivalent for a tournament $T_n$ on $n$ vertices.\\
(i) $T_n$ is locally transitive.\\
(ii) $T_n$ does not contain a $4$ vertex tournament that contains a triangle and a vertex that either dominates or is dominated by the vertices of the triangle, that is, a $4$ vertex tournament that is not locally transitive.\\
(iii) For the skew-symmetric matrix $M(T_n)$ we have
$$\varphi_{M(T_n)}(x)=\frac{1}{2}((x+1)^n+(x-1)^n).$$
(iv) For the skew-symmetric matrix $M(T_n)$ we have
$\det(M(T)_{S,S})=1$ for each $S\subseteq [n]$ with $|S|$ even.
\end{Lemma}

\begin{proof}
The equivalence of $(i)$ and $(ii)$ is trivial as a non-transitive tournament always contains a triangle. 

First note that
$$\frac{1}{2}((x+1)^n+(x-1)^n)=\sum_{k=0}^{n/2}\binom{n}{2k}x^{n-2k}.$$

To prove the equivalence of $(iii)$ and $(iv)$, observe that
$$\varphi_{M(T_n)}(x)=\sum_{S\subseteq [n]}x^{n-|S|}(-1)^{|S|}\det(M(T)_{S,S}).$$
By Lemma \ref{skew-symmetric} $(iii)$ and $(iv)$ and the proof of Theorem~\ref{lower_bound_K_n},
$\det(M(T)_{S,S})=0$ if $|S|$ is odd, and $\det(M(T)_{S,S})\geq 1$ is $|S|$ is even. So the coefficient of $x^{n-2k}$ is at least $\binom{n}{2k}$ and equality holds if and only if $\det(M(T)_{S,S})=1$ for all set $S$ of size $2k$. 

Finally, let us prove the equivalence of $(ii)$ and $(iv)$.
It turns out that for the tournaments on $4$ vertices that are not locally transitive, the determinant of the skew-symmetric adjacency matrix is $9$ and for all other tournaments on $4$ vertices, the determinant is $1$.  
This shows that for sets with $|S|=4$, $\det(M(T_n)_{S,S})=1$ if any only if $(ii)$ holds.

Next we show that if for all $|S|=4$ we have $\det(M(T_n)_{S,S})=1$, then we also have $\det(M(T_n)_{S,S})=1$ whenever $|S|$ is even. We prove this statement by induction on the size of $|S|=2r$. Suppose that $|S|\geq 6$ and we already know the statement for all $R\subset S$. Then
$$\varphi_{M(T)_{S,S}}(x)=x^{2r}+\binom{2r}{2}x^{2r-2}+\dots +\binom{2r}{2r-4}x^4+\binom{2r}{2r-2}x^2+\det(M(T)_{S,S}).$$
Note that
$$\varphi_{M(T)_{S,S}}(x)=\prod_{j=1}^{r}(x-i\alpha_j)(x+i\alpha_j)=\prod_{j=1}^{r}(x^2+\alpha_j^2)$$
as $M(T)_{S,S}$ is a skew-symmetric matrix itself. This means that
$$x^{r}+\binom{2r}{2}x^{r-1}+\dots +\binom{2r}{2r-4}x^2+\binom{2r}{2r-2}x+\det(M(T)_{S,S})=\prod_{j=1}^{r}(x+\alpha_j^2),$$
that is, it is a real-rooted polynomial. For a real-rooted polynomial $\sum_{j=0}^da_jx^j$ with non-negative coefficients Newton's inequality says that for $1\leq t\leq d-1$ we have
$$\frac{a_{t-1}}{\binom{d}{t-1}}\cdot \frac{a_{t+1}}{\binom{d}{t+1}}\leq \left(\frac{a_{t}}{\binom{d}{t}}\right)^2.$$
Let us apply this inequality for the above polynomial and $t=1$:
$$\frac{\binom{2r}{4}}{\binom{r}{2}}\cdot \det(M(T)_{S,S})\leq \left(\frac{\binom{2r}{2}}{\binom{r}{1}}\right)^2.$$
This gives that
$$\det(M(T)_{S,S})\leq 3\cdot \frac{2r-1}{2r-3}.$$
Since $r\geq 3$ we get that $\det(M(T)_{S,S})\leq 5$. On the other hand, $\det(M(T)_{S,S})$ is an odd number, moreover, it is the square of the Pfaffian (by Lemma \ref{skew-symmetric} $(iv)$) so if $\det(M(T)_{S,S})>1$, then it is at least $9$. Thus, $\det(M(T)_{S,S})=1$ for every even-size $S$. This completes the proof of the equivalence of the four conditions.
\end{proof}

Next we prove Theorem~\ref{Eulerian_lower_bound_K_n}.

\begin{proof}[Proof of Theorem~\ref{Eulerian_lower_bound_K_n}]

If $T$ is an Eulerian tornament, then all degrees are $\frac{n-1}{2}$ and we immediately get that
\begin{align*}
\arb(O) \geq \frac{1}{n^2}\left(\left(\frac{n+1}{2}\right)^{n} + \left(\frac{n-1}{2}\right)^{n}\right).
\end{align*}
It is also clear that the given tournament is locally transitive, and so achieves the lower bound. The uniqueness follows from the following result of Huang \cite{huang1992tournament} characterising locally transitive digraphs (see also the paper \cite{aboulker2022decomposing}).

\begin{Th}[Huang \cite{huang1992tournament}]
If $D=(V,E)$ is a simple connected digraph. Then the following two conditions are equivalent:\\
(i) $D$ is locally transitive.\\
(ii) There exists a cyclic ordering of the vertices (say clockwise) such that if $(u,v) \in E$, then for every vertex $w$ between $u$ and $v$ we have $(u,w)\in E$ and $(w,v)\in E$.
\end{Th}

From this theorem we immediately see that given the out-degree sequence of locally transitive tournament, the vertices have a cyclic ordering such that for every vertex $v$ the out-neighbor set $N^+(v)$ is simply the next $d^+(v)$ vertices in the cyclic order. In particular, this shows the uniqueness of the minimizing Eulerian tournament.
\end{proof}

Below we give another proof that the swirl tournament achieves the minimum number of arborencences. This proof is more direct and has the advantage that it connects the eigenvalues of $M(T)$ with the number of arborencences. 

First we prove a lemma that connects the number of arborencenses with the characteristic polynomial of $M(T)$.

\begin{Lemma}\label{lem:skew_adj_of_K_n}
Let $T$ be a $d$-regular tournament on $n=2d+1$ vertices. Let $M(T)$ be the skew-symmetric adjacency matrix of $T$. If the eigenvalues of $M(T)$ are $\alpha_1,\dots ,\alpha_{n-1},\alpha_n=0$, then the eigenvalues of $L(T)$ are $\frac{n-\alpha_1}{2},\dots ,\frac{n-\alpha_{n-1}}{2}$ and $0$. Furthermore,
$$\arb(T)=\frac{1}{n^2 2^{n-1}}\varphi_{M(T)}(n).$$
\end{Lemma}

\begin{proof}
Note that we have $L(T)=\frac{1}{2}(nI-J-M(T))$. We have $M(T)\underline{1}=\underline{0}$ and $L(T)\underline{1}=\underline{0}$. Real skew-symmetric matrices are diagonalizable (since they are normal), that is, there are eigenvectors $\underline{v}_1,\dots ,\underline{v}_{n-1},\underline{v}_n=\underline{1}$ of $M(T)$ belonging to the eigenvalues $\alpha_1,\dots,\alpha_n=0$ that form a basis. 
Note that $\underline{1}^TM(T)=\underline{0}$, which implies that $\underline{1}$ and $\underline{v}_k$ are orthogonal whenever $\lambda_k\neq 0$. (Indeed, $\underline{0}=(\underline{1}^TM(T))\underline{v}_k =\underline{1}^T(M(T)\underline{v}_k)=\lambda_k\underline{1}^T\underline{v}_k$.) 
If $\lambda_k=0$, then we can choose an othogonal eigenbasis from the eigensubspace belonging to the eigenvalue $0$, so we can assume that $\underline{1}$ and $\underline{v}_k$ are orthogonal in this case, too.

Then for $1\leq k\leq n-1$, we have
$$L(T)\underline{v}_k=\frac{1}{2}(nI-J-M(T))\underline{v}_k=\frac{n-\alpha_k}{2}\underline{v}_k$$
since $J\underline{v}_k=\underline{0}$ as $\underline{1}$ and $\underline{v}_k$ are orthogonal. This proves the first part of the claim.

To prove the second part, observe that
$$\arb(T)=\frac{1}{n}\prod_{i=1}^{n-1}\frac{n+\alpha_i}{2}=\frac{1}{n^2 2^{n-1}}\cdot n\prod_{i=1}^{n-1}(n+\alpha_i)=\frac{1}{n^2 2^{n-1}}\cdot n\prod_{i=1}^{n-1}(n-\alpha_i)=\frac{1}{n^2 2^{n-1}}\varphi_{M(T)}(n).$$
In the third equality, we used the fact that the eigenvalues of $M(T)$ come in pairs $\alpha,-\alpha$, see Lemma \ref{skew-symmetric} $(ii)$.
This proves the second part of the claim. 
\end{proof}

Next, we compute the characteristic polynomial of $M(T)$ for the swirl tournament $T$.

We claim that
$$\varphi_{M(T)}(x)=\frac{1}{2}((x+1)^n+(x-1)^n).$$
$M(T)$ is a circulant matrix so its eigenvectors are of the form $(1,\varepsilon,\varepsilon^2,\dots ,\varepsilon^{n-1})$, where $\varepsilon^n=1$. 

When $\varepsilon=1$ we get the usual $\underline{1}$ eigenvector of $M(T)$ with eigenvalue $0$. Otherwise we get the eigenvalue
$$\sum_{k=1}^d\varepsilon^k- \sum_{k=d+1}^{n-1}\varepsilon^k=(1-\varepsilon^d)\varepsilon\frac{1-\varepsilon^d}{1-\varepsilon}=\frac{\varepsilon-2\varepsilon^{d+1}+1}{1-\varepsilon}.$$
Since $n$ is odd, there is a unique $n$-th root of unity $\eta\neq 1$ for which $\eta^2=\varepsilon$. Then 
$$\frac{\varepsilon-2\varepsilon^d+1}{1-\varepsilon}=\frac{\eta^2-2\eta^{2(d+1)}+1}{1-\eta^2}=\frac{\eta^2-2\eta+1}{1-\eta^2}=\frac{1-\eta}{1+\eta}$$
whence
$$\varphi_{M(T)}(x)=x\prod_{\eta^n=1 \atop \eta\neq 1}\left(x-\frac{1-\eta}{1+\eta}\right).$$
We claim that for odd $n$ this is nothing else than 
$$\frac{1}{2}((x+1)^n+(x-1)^n).$$
It is clear that both polynomials are monic and that $0$ is a root of both polynomials. If $\alpha\neq 0$ is a root of the latter polynomial, then $\alpha\neq -1$, and so
$$\frac{1}{2}((\alpha+1)^n+(\alpha-1)^n)=\frac{1}{2}(\alpha+1)^n\left(1+\left(\frac{\alpha-1}{\alpha+1}\right)^n\right).$$
Since $n$ is odd we get that $\frac{\alpha-1}{\alpha+1}=-\eta$ for some $\eta\neq 1$ for which $\eta^n=1$.  
We get 
$\alpha=\frac{1-\eta}{1+\eta}$, proving that the two polynomials are equal. This completes the proof that the swirl tournament achieves the lower bound 
for the number of arborescences of Eulerian tournaments.
\bigskip

Next we prove Theorem~\ref{transitive_tournament}.

\begin{proof}[Proof of Theorem~\ref{transitive_tournament}]
For the transitive tournament, $L(TR_n)$ is an upper triangular matrix, so $\det(L(TR_n)_{n,n})=(n-1)!$ (the product of the values in the main diagonal). 

We show that for any tournament with outdegree sequence $d_1^+, \dots d_k^+$, we have
$$\prod_{k=1}^n(d^+_k+1)\geq n!.$$
The out-degree sequence of the transitive tournament is $0,1,\dots ,n-1$, which achieves the bound, and this degree sequence determines the transitive tournament. Suppose for contradiction that some tournament $T_n$ with a different degree sequence $d_1^+, \dots, d_n^+$ minimizes $\prod_{k=1}^n(d^+_k+1)$. Then $T_n$ must have two vertices $u$ and $v$ with the same out-degree, say $d^+(u)=d^+(v)=d$. Now flip the orientation of the edge between $u$ and $v$. Suddenly, one of the them has out-degree $d+1$, the other one has $d-1$. No other out-degree changed. This implies that $\prod_{k=1}^n(d^+_k+1)$ decreased as $(d-1+1)(d+1+1)<(d+1)^2$, contradicting the assumption that $T_n$ minimizes the quantity $\prod_{k=1}^n(d^+_k+1)$. So 
$$\allarb(T_n)\geq \frac{1}{n}\left(\prod_{k=1}^n(d^+_k+1)+\prod_{k=1}^nd^+_k\right)\geq \frac{1}{n}n!=(n-1)!=\allarb(TR_n).$$
It is also clear from the proof that equality only holds if $T_n$ is isomorphic to $TR_n$, otherwise one can strictly decrease $\prod_{k=1}^n(d^+_k+1)$.
\end{proof}

\subsection{Minimizing orientation for complete bipartite graphs}

In this section we prove Theorem~\ref{lower_bound_K_nm}. The proof of Theorem~\ref{lower_bound_K_nm} is very similar to the proof of Theorem~\ref{lower_bound_K_n}. In fact, it is a bit simpler.

\begin{proof}[Proof of Theorem~\ref{lower_bound_K_nm}]
Let $D$ be an orientation of $K_{n,m}$. Then the Laplacian matrix of $D$ looks as follows:
$$L(D)=\left(\begin{array}{cc}
\frac{m}{2}I_n & -A_1 \\
-A_2 & \frac{n}{2}I_m 
\end{array}
\right),
$$
where $A_1$ and $A_2$ are matrices of size $n\times m$ and $m\times n$, respectively. Now let us consider the matrix 
$$S(D)=\left(\begin{array}{cc}
\frac{m}{2}I_n & \frac{1}{2}J_{n,m}-A_1 \\
\frac{1}{2}J_{m,n}-A_2 & \frac{n}{2}I_m 
\end{array}
\right),
$$
where $J_{n,m}$ and $J_{m,n}$ are the matrices of all $1$'s of size $n\times m$ and $m\times n$, respectively.
Let us consider the following $4$ vectors in $\mathbb{C}^{n+m}$:
$$\underline{v}_1=(1,1,\dots ,1)\ \ \text{and}\ \ 
 \underline{v}_2=\left(1,\dots ,1,-\frac{n}{m},\dots ,-\frac{n}{m}\right)$$
$$\underline{v}_3=\left(1,\dots ,1,0,\dots ,0\right)\ \ \text{and}\  \ \underline{v}_4=\left(0,\dots ,0,1,\dots ,1\right),$$
where the first $n$ coordinates and the last $m$ coordinates are equal. Observe that
$$L(D)\underline{v}_1=\underline{0},\ \ L(D)\underline{v}_2=\frac{n+m}{2}\underline{v}_2,\ \ S(D)\underline{v}_3=\frac{m}{2}\underline{v}_3,\ \ S(D)\underline{v}_4=\frac{n}{2}\underline{v}_4.$$
Clearly, $\langle \underline{v}_1,\underline{v}_2\rangle=\langle \underline{v}_3,\underline{v}_4\rangle$, let us denote this $2$-dimensional vector space by $V_2$. This is an invariant subspace for both $L(D)$ and $S(D)$. If $\underline{v}\in V_2^{\perp}$, then the sum of the first $n$ coordinates of $\underline{v}$ and the sum of the last $m$ coordinates of $\underline{v}$ are both $0$ implying that 
$L(D)\underline{v}=S(D)\underline{v}$. This shows that the remaining $n+m-2$ eigenvalues of $L(D)$ and $M(D)$ are the same, let us denote the multiset of these eigenvalues by $R$. Indeed, if we take the vectors $\underline{v}_3$ and $\underline{v}_4$ and we extend it to a basis of $\mathbb{C}^{n+m}$ by taking a basis of $V_2^{\perp}$, then in this basis both $L(D)$ and $S(D)$ will have the form $\left(\begin{array}{cc} M_{11} & M_{12} \\ 0 & M_{22}\end{array}\right)$, where $M_{11}$ is a $2\times 2$ matrix describing the action of $L(D)$ and $S(D)$ on $V_2$ and $M_{12}$ and $M_{22}$ are the same for $L(D)$ and $S(D)$ as for $\underline{v}\in V_2^{\perp}$ we have $L(D)\underline{v}=S(D)\underline{v}$. From this it follows that $R$ is the multiset of eigenvalues of $M_{22}$. 

From Corollary~\ref{cor:product_of_nonzero_eigenvalues_eulerian} and the known eigenvalues corresponding to $v_1, v_2, v_3$ and $v_4$ we get that
$$\arb(D)=\frac{1}{n+m}\cdot \frac{n+m}{2}\cdot \prod_{\lambda\in R}\lambda=\frac{1}{n+m}\cdot \frac{n+m}{2}\cdot  \frac{2}{n}\cdot \frac{2}{m}\cdot \det(S(D))=\frac{2}{nm}\det(S(D)).$$
Next observe that $\left(\frac{1}{2}J_{n,m}-A_1\right)^T=-\left(\frac{1}{2}J_{n,m}-A_2\right)$, so 
$$S(D)=\left(\begin{array}{cc}
\frac{m}{2}I_n & B \\
-B^T & \frac{n}{2}I_m 
\end{array}
\right),
$$
Thus
$$\det(S(D))=\sum_{k=0}^{\min(n,m)}\left(\frac{m}{2}\right)^{n-k}\left(\frac{n}{2}\right)^{m-k}\sum_{|S|=|T|=k}\det(B_{S,T})^2.$$
For $k=1$ we have $\sum_{|S|=|T|=k}\det(B_{S,T})^2=\frac{nm}{4}$ as all elements of $B$ are $\pm \frac{1}{2}$. Since all terms are non-negative, we get that 
$$\det(S(D))\geq \left(\frac{m}{2}\right)^{n}\left(\frac{n}{2}\right)^{m}+\frac{nm}{4}\left(\frac{m}{2}\right)^{n-1}\left(\frac{n}{2}\right)^{m-1}=2\left(\frac{m}{2}\right)^{n}\left(\frac{n}{2}\right)^{m},$$
that is
$$\arb(D)\geq \left(\frac{m}{2}\right)^{n-1}\left(\frac{n}{2}\right)^{m-1}.$$
Note that for $B=\frac{1}{2}\left(\begin{array}{cc} J_{n/2,m/2} & -J_{n/2,m/2} \\ -J_{n/2,m/2} & J_{n/2,m/2} \end{array}\right)$ we have $\det(B_{S,T})=0$ whenever $|S|=|T|\geq 2$ since this is a rank $1$ matrix. So we have equality for this matrix. This matrix corresponds exactly to the directed graph described in the theorem. 

Next we show that the minimizing orientation is unique up to isomorphism. Indeed, $B$
must be a rank $1$ matrix as otherwise there would be sets $S$ and $T$ such that $|S|=|T|=2$ and $\det(B_{S,T})>0$ showing that $\arb(D)>\left(\frac{m}{2}\right)^{n-1}\left(\frac{n}{2}\right)^{m-1}$. Since $B$ is rank $1$ there are vectors $\underline{u}$ and $\underline{v}$ such that $B=\frac{1}{2}\underline{u}\underline{v}^T$. By replacing $\underline{u}$ and $\underline{v}$ with $\alpha\underline{u}$ and $\frac{1}{\alpha}\underline{v}$ for some $\alpha \neq 0$ we can assume that $\underline{u}_1=1$. Since all entries of $2B_{ij}$ are $\pm 1$ we immediately get that all entries of $\underline{u}$ and $\underline{v}$ are $\pm 1$. Since each row of $2B$ contain $\frac{m}{2}$ and and each column contain $\frac{n}{2}$ elements that are $1$ and $-1$ we get that it is true for both $\underline{u}$ and $\underline{v}$. This means that $2B$ is isomorphic to $\left(\begin{array}{cc} J_{n/2,m/2} & -J_{n/2,m/2} \\ -J_{n/2,m/2} & J_{n/2,m/2} \end{array}\right)$ up to the permutations of rows and columns. In other words, any minimizing orientation is isomorphic to the one described in the theorem. 
\end{proof}

Let us mention that if $n=m$ we can also describe the eigenvalues of $L(D)$. We do not detail the proof as it is practically the same as the first part of the above proof.

\begin{Lemma}
Let $D$ be an orientation of $K_{n,n}$, where $n$ is even. Let $M(D)$ be its skew-symmetric adjacency matrix, and let $L(D)$ be its Laplacian matrix. Let $\underline{1}$ be the all-$1$ vector of length $2n$, and let $\underline{j}$ be the vector of length $2n$ whose first $n$ coordinates are $1$, and last $n$ coordinates are $-1$. Then
$$L(D)\underline{1}=\underline{0},\ \  L(D)\underline{j}=n\underline{j},\ \ M(D)\underline{1}=\underline{0},\ \ M(D)\underline{j}=\underline{0}.$$
Furthermore, if $\underline{v}$ is an eigenvector of $M(D)$ corresponding to eigenvalue $\alpha$, that is orthogonal to the vectors $\underline{1}$ and $\underline{j}$, then $v$ is an eigenvector of $L(D)$ corresponding to eigenvalue $\frac{n+\alpha}{2}$. In particular,
$$\arb(D)=\frac{1}{2^{2n-1}n^2}\varphi_{M(D)}(n).$$
\end{Lemma}

\subsection{Proof of Theorem~\ref{thm:double_graph_minimizer}}

In this section we prove Theorem~\ref{thm:double_graph_minimizer}. The proof is based on the following result of Ostrowski and Taussky. This inequality can be found in \cite{horn2012matrix} as Theorem 7.8.19. 

\begin{Lemma}[Ostrowski and Taussky]\label{conj:symmetric}
Let $A$ be a real square matrix. Assume its symmetric part $\frac{A+A^T}{2}$ is positive-definite. Then,
\begin{align*}
    \det\left(\frac{A+A^T}{2}\right) \leq \det(A)
\end{align*}
with equality if and only if $\frac{A+A^T}{2}=A$.
\end{Lemma}

The following theorem implies Theorem~\ref{thm:double_graph_minimizer}

\begin{Th}\label{thm:sp_tree_lower_bound_for_arb}
Let $G$ be an undirected connected graph with an Eulerian orientation $O$, and let $\sp(G)$ denote the number of spanning trees of the graph $G$. Then
\begin{align*}
    \arb(O) \geq \frac{1}{2^{n-1}}\sp(G)
\end{align*}
with equality if and only if $O$ is the symmetric orientation. 
\end{Th}

\begin{proof}
Since $O$ is Eulerian, we can see that $L(O)+L(O)^T = L(G)$, where $L(G)$ is the Laplacian matrix of the undirected graph, and $L(O)$ is the Laplacian matrix of the directed graph $O$. This is because
\begin{align*}
    L(O)+L(O)^T = (D(O)+D(O)^T)-(A(O)+A(O)^T) = D(G) - A(G) = L(G)
\end{align*}
where we use the fact that $O$ is Eulerian to deduce $D(O)+D(O)^T=D(G)$. As usual let $L(O)_{\overline{n},\overline{n}}$ be the matrix obtained from $L(O)$ by deleting the $n^{th}$ row and column. Then we have $L(O)_{\overline{n},\overline{n}}+L(O)_{\overline{n},\overline{n}}^T=L(G)_{\overline{n},\overline{n}}$. Note that $L(G)$ is a symmetric matrix because $G$ is undirected. Furthermore $L(G)$ is positive-definite since $G$ is connected. This means that $L(G)_{\overline{n},\overline{n}}$ is positive-definite. This means $\frac{L(O)_{\overline{n},\overline{n}}+L(O)_{\overline{n},\overline{n}}^T}{2}=\frac{1}{2}L(G)_{\overline{n},\overline{n}}$ is also positive-definite. From Lemma \ref{conj:symmetric} this implies that
$$\det\left(\frac{L(O)_{\overline{n},\overline{n}}+L(O)_{\overline{n},\overline{n}}^T}{2}\right)\leq \det (L(O)_{\overline{n},\overline{n}}).$$
Therefore we have,
\begin{align*}
\frac{1}{2^{n-1}}\sp(G) &= \frac{1}{2^{n-1}}\det(L(G)_{\overline{n},\overline{n}}) \\
&= \det\left(\frac{1}{2}L(G)_{\overline{n},\overline{n}}\right)\\
&= \det\left(\frac{L(O)_{\overline{n},\overline{n}}+L(O)_{\overline{n},\overline{n}}^T}{2}\right)\\
&\leq \det(L(O)_{\overline{n},\overline{n}})\\
&= \arb(O).
\end{align*}
We have equality in the above if and only if $\det\left(\frac{L(O)_{\overline{n},\overline{n}}+L(O)_{\overline{n},\overline{n}}^T}{2}\right)=\det(L(O)_{\overline{n},\overline{n}})$. From lemma \ref{conj:symmetric} this happens if and only if $\frac{L(O)_{\overline{n},\overline{n}}+L(O)_{\overline{n},\overline{n}}^T}{2}=L(O)_{\overline{n},\overline{n}}$. Because $\frac{L(O)_{\overline{n},\overline{n}}+L(O)_{\overline{n},\overline{n}}^T}{2}=\frac{1}{2}L(G)_{\overline{n},\overline{n}}$ this happens if and only if $\frac{1}{2}L(G)_{\overline{n},\overline{n}}=L(O)_{\overline{n},\overline{n}}$. Finally, since the rows and columns must sum to zero, this is further equivalent to $\frac{1}{2}L(G)=L(O)$. Which means we have equality if and only if $O$ is the symmetric orientation.
\end{proof}

\begin{proof}[Proof of Theorem~\ref{thm:double_graph_minimizer}]
    Immediate from Theorem \ref{thm:sp_tree_lower_bound_for_arb}.
\end{proof}

\begin{Rem}
There is another connection between the number of arborescences and the number of spanning trees. A simple double counting argument shows that if $G$ is a graph on $n$ vertices and $m$ edges, then
$$\frac{1}{2^m}\sum_{O}\mathrm{allarb}(O)=\frac{n}{2^{n-1}}\sp(G),$$
where the summation is for all orientations, not just the Eulerian ones.
\end{Rem}

\section{Upper bounds} \label{sect:upper_bounds}

In this section we prove Theorems~\ref{conj:upper_bound} and \ref{trivial_upper_bound}.

\subsection{General upper bound for digraphs.}

\begin{proof}[Proof of Theorem~\ref{conj:upper_bound}]
Let $L=L(D)$ be the Laplacian, and $A=A(D)$ be the adjacency matrix of $D$. Let the eigenvalues of $L$ be $\lambda_1,\ldots,\lambda_n=0$. Then applying the geometric-quadratic mean inequality we have,
$$\prod_{i=1}^{n-1}\lambda_i= \prod_{i=1}^{n-1} |\lambda_i|\leq \left(\frac{1}{n-1}\sum_{i=1}^{n-1} |\lambda_i|^2\right)^{(n-1)/2}.$$
Using the fact that $\sum_{i=1}^n |\lambda_i|^2\leq ||L||_F^2$ (Lemma~\ref{Frobenius}) we have,
\begin{align*}
\left(\frac{1}{n-1}\right)^{(n-1)/2}\left(\sum_{i=1}^n |\lambda_i|^2\right)^{(n-1)/2}\leq \left(\frac{1}{n-1}\right)^{\frac{n-1}{2}}\left(||L||_F^2\right)^{(n-1)/2}.
\end{align*}
On the other hand,
\begin{align*}
||L||_F^2 = \sum_{i,j}L_{ij}^2 = \sum_{i}L_{ii}^2 + \sum_{i,j}A_{ij}^2 = \sum_{i=1}^n (d^+_i)^2 + m 
\end{align*}
where we use the fact that $D$ is simple.
Altogether,
$$\allarb(D)= \prod_{i=1}^{n-1}\lambda_i
\leq \left(\frac{1}{n-1}\right)^{\frac{n-1}{2}}(||L||_F^2)^{(n-1)/2}
= \left(\frac{1}{n-1}\right)^{\frac{n-1}{2}}\left(\sum_{i=1}^n (d^+_i)^2 + m\right)^{(n-1)/2}$$
which is the desired inequality.

If $O$ is an Eulerian orientation of the simple graph $G$ with degree sequence $d_1, \dots , d_n$, then $d^+_i=\frac{d_i}{2}$ and $\allarb(O) = n\cdot \arb(O)$, whence the second inequality follows.
\end{proof}

\begin{Cor}\label{upper_bound_max_degree}
Let $D$ be a simple digraph on $n$ vertices and $m$ edges. If $\Delta^+$ is the maximum out-degree of vertices in $D$, then we have
\begin{align*}
    \allarb(D) &\leq n^{\frac{n-1}{2}}\left(\frac{1}{n-1}\right)^{\frac{n-1}{2}}\left((\Delta^+)^2 + \Delta^+ \right)^{\frac{n-1}{2}}.
\end{align*}
In particular if $G$ is a simple Eulerian graph, and $\Delta$ is the maximum degree of vertices in $G$, and $O$ is an Eulerian orientation of $G$, then we have
\begin{align*}
    \arb(O) \leq n^{\frac{n-3}{2}}\left(\frac{1}{n-1}\right)^{\frac{n-1}{2}}\left(\frac{1}{2}\right)^{n-1}\left(\Delta^2+2\Delta \right)^{\frac{n-1}{2}}.
\end{align*}
\end{Cor}
\begin{proof}
    This is immediate from Theorem \ref{conj:upper_bound}. For the first inequality we have $d_i^+\leq\Delta^+$ which means $\sum_{i=1}^n (d_i^+)^2 \leq n(\Delta^+)^2$ and $m\leq n\Delta^+$, then the upper bound follows from Theorem \ref{conj:upper_bound}. For the second inequality we have $d_i\leq\Delta$ which means $\sum_{i=1}^n d_i\leq n\Delta^2$ and $m\leq \frac{1}{2}n\Delta$, then the upper bound follows from Theorem \ref{conj:upper_bound}.
\end{proof}

\begin{proof}[Proof of Theorem~\ref{upper_bound_for_K_n}]
Since $d_i=n-1$ and $m=\binom{n}{2}$, the bound in Theorem \ref{conj:upper_bound} implies
\begin{align*}
    \arb(O) &\leq \frac{1}{n}\left(\frac{1}{n-1}\right)^{\frac{n-1}{2}}\left(\frac{1}{4}\sum_{i=1}^{n}d_i^2 + m\right)^{\frac{n-1}{2}}\\
    &= \frac{1}{n}\left(\frac{1}{n-1}\right)^{\frac{n-1}{2}}\left(\frac{n(n-1)^2}{4} + n\frac{n-1}{2}\right)^{\frac{n-1}{2}}\\
    &= \frac{1}{n}\left(\frac{n(n+1)}{4}\right)^{\frac{n-1}{2}}
\end{align*}
which is the desired inequality.

Next, we show that Hadamard tournaments attain the upper bound. Suppose that $O$ is an Hadamard tournament. The definition of Hadamard tournament implies that 
$A_{i,i}=\frac{n-1}{2}$. Thus, $O$ is Eulerian. 
By \cite[Proposition 3.1]{de1992algebraic}, the adjacency matrix $A:=A(O)$ has an eigenvalue $n$ with multiplicity $1$ and eigenvalues $\frac{-1}{2}+ i\frac{\sqrt{n}}{2}$ and $\frac{-1}{2}- i\frac{\sqrt{n}}{2}$ each with multiplicity $\frac{n-1}{2}$. Therefore its Laplacian $L:=L(O)$ has an eigenvalue $0$ with multiplicity $1$ and eigenvalue $\frac{n}{2}\pm i\frac{\sqrt{n}}{2}$ each with multiplicity $\frac{n-1}{2}$. By Corollary \ref{cor:product_of_nonzero_eigenvalues_eulerian},
\begin{align*}
    \arb(O) = \frac{1}{n}\prod_{\lambda_i\neq 0}\lambda_i = \frac{1}{n}\left(\frac{n(n+1)}{4}\right)^{\frac{n-1}{2}}
\end{align*}
as required. 

Now suppose $\arb(O)=\frac{1}{n}\left(\frac{n(n+1)}{4}\right)^{\frac{n-1}{2}}$. 
Then in the proof of Theorem \ref{conj:upper_bound}, we need equality in the geometric mean- quadratic mean inequality for $|\lambda_1|, \dots , |\lambda_{n-1}|$, which implies $|\lambda_1|=\dots =|\lambda_{n-1}|$, thus $|\lambda_i|=\sqrt{\frac{n(n+1)}{4}}$ for $i=1, \dots n-1$.

By Lemma~\ref{lem:skew_adj_of_K_n} we also know that the real part of all non-zero eigenvalues are $\frac{n}{2}$ as the eigenvalues of the skew-symmetric adjacency matrix $M(T)$ are purely complex according to Lemma~\ref{skew-symmetric}. This implies that all eigenvalues are $\frac{n\pm i\sqrt{n}}{2}$ and since they are the eigenvalues of a real matrix, the multiplicities of both numbers are $\frac{n-1}{2}$. So $T$ has exactly three eigenvalues. Because $\arb(O)>0$, we know $O$ has an Eulerian tour due to the BEST theorem, and thus is strongly connected. Since $O$ is a strongly connected tournament and has exactly $3$ distinct eigenvalues, it follows that $O$ is a Hadamard tournament by \cite[Theorem 3.2]{de1992algebraic}.
\end{proof}

\begin{Rem}
The bound in Theorem~\ref{upper_bound_for_K_n} is not tight for all $n$, for example when $n=5$ the bound gives an upper bound of $11.25$. But when $n$ is prime and $n\equiv 3\ (\bmod 4)$, then the bound is attained by the Paley tournament on $n$ vertices. So the bound is tight for an infinite family of orientations.
\end{Rem}

\begin{Rem}
In general, it is not true that an Eulerian orientation maximizes $\mathrm{allarb}(O)$ when there is such an orientation. For instance, for the graph $K_{2,4}$ the Eulerian orientation $O$ --which is unique up to isomorphism-- gives $\mathrm{allarb}(O)=12$, whereas there is an orientation $O'$ for which $\mathrm{allarb}(O')=16$. Nevertheless, we conjecture that $\allarb$ is maximized by an Eulerian orientation for $K_n$ when $n$ is odd, see Conjecture~\ref{max_orientation_K_n}.
\end{Rem}

\subsection{Trivial upper bound.} 

\begin{proof}[Proof of Proposition~\ref{trivial_upper_bound}.]
Once we fix the root to $v_k$, any arborescence rooted at $v_k$ uses exactly one of the $d^+_j$ edges at vertex $j$. Hence the number of such arborescenses is at most $\prod_{j\neq k}d^+_j$, and the number of all arborescenses is at most
$$\mathrm{allarb}(T)\leq \sum_{k=1}^n\prod_{j\neq k}d_j^+<\prod_{k=1}^n(d^+_k+1).$$
\end{proof}

\begin{Rem}\label{rem:ratio} The bounds in Theorem~\ref{lower_bound_K_n}
and Proposition~\ref{trivial_upper_bound} have ratio less than $n$. Probably an even stronger upper bound is true since for Eulerian tournaments the ratio of the upper and lower bounds provided by Theorem~\ref{upper_bound_for_K_n} and \ref{Eulerian_lower_bound_K_n} are within constant factor:
$$\frac{\left(\frac{n(n+1)}{4}\right)^{(n-1)/2}}{\frac{1}{n}\left(\left(\frac{n+1}{2}\right)^n+\left(\frac{n-1}{2}\right)^n\right)}\leq \frac{2e^{3/2}}{e^2+1}\approx 1.06846.$$

\end{Rem}

\section{Concluding remarks and questions}

There is a general intuition that guides in the solution of many Eulerian minimization problems: the orientation minimizing the number of arborescences is the one that has many short directed cycles. Similar intuition appears at many different problems: For the number of spanning trees this is justified by McKay \cite{mckay1983spanning}. In case of Eulerian orientations, the situation is opposite: short cycles tend to increase their number \cite{isaev2024correlation}. In case of Eulerian tours this phenomenon was also observed by Creed (see page 154 of \cite{creed2010counting}: ``a strong connection between the number of short cycles of different lengths and the number of Eulerian tours of graphs'').

In our paper the short cycle phenomenon is less apparent for the complete graph $K_n$, but both for complete bipartite graphs $K_{n,m}$ and double graphs, the minimizing Eulerian orientations are the ones that contain the most directed cycles among cycles of minimal length. The following conjecture is also related to the intuition on many short directed cycles. As we explain below, this conjecture is the special case of \cite[Conjecture 5.6]{sym_ribbon}. 

\begin{Conj}\label{conj:planar}
Let $G$ be an Eulerian planar graph. The orientation minimizing the number of arborenscences is the one that is alternatingly in- and outward oriented at each vertex, that is, where the cycle around each face is oriented.
\end{Conj}
\begin{figure}[H]
	\begin{tikzpicture}[scale=.8]
	\node [circle,fill,scale=.6,draw] (1) at (0,0) {};
	\node [circle,fill,scale=.6,draw] (2) at (1,1.3) {};
	\node [circle,fill,scale=.6,draw] (3) at (-1,1.3) {};
	\node [circle,fill,scale=.6,draw] (4) at (-2,2.6) {};
	\node [circle,fill,scale=.6,draw] (5) at (0,2.6) {};	
	\node [circle,fill,scale=.6,draw] (6) at (2,2.6) {};	
	\path [thick,->,>=stealth] (1) edge [left] node {} (2);
	\path [thick,->,>=stealth] (2) edge [above] node {} (6);
	\path [thick,->,>=stealth] (6) edge [below] node {} (5);
	\path [thick,->,>=stealth] (5) edge [left] node {} (4);
	\path [thick,->,>=stealth] (4) edge [above] node {} (3);
	\path [thick,->,>=stealth] (3) edge [below] node {} (1);
	\path [thick,->,>=stealth] (2) edge [above] node {} (3);
	\path [thick,->,>=stealth] (5) edge [below] node {} (2);
	\path [thick,->,>=stealth] (3) edge [left] node {} (5);
	\end{tikzpicture}
	\caption{An Eulerian plane graph and its alternating orientation.} 
 \label{fig:bipartite_graph}
\end{figure}
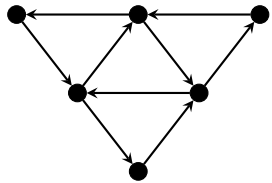

\begin{Rem}
As we remarked in Section \ref{ss:geom_motivation}, the problems considered in this paper are related to a problem concerning volumes of facets of symmetric edge polytopes. Conjecture \ref{conj:planar} can be translated to the symmetric edge polytope language as follows: take the planar dual $G^*$ of the planar Eulerian graph $G$. Since $G$ was Eulerian, $G^*$ is bipartite (say, with partite classes $A$ and $B$). The planar dual of the orientation mentioned in Conjecture \ref{conj:planar} is the orientation of $G^*$ where each edge points from $A$ to $B$. In \cite{sym_ribbon}, this is called a standard orientation of $G^*$.

The facets of the symmetric edge polytope of $G^*$ correspond to planar duals of Eulerian orientations of $G$, and their normalized volume is the arborescence number of these Eulerian orientations. The duals of the Eulerian orientations are exactly those orientations of $G^*$, where for each cycle, the number of edges in the two cyclic directions agree. (In \cite{sym_ribbon}, these are called semi-balanced orientations.) 
Conjecture \ref{conj:planar} says that among these orientations, the standard orientation corresponds to a facet of minimal volume.

In \cite[Conjecture 5.6]{sym_ribbon}, it is further conjectured that the standard orientation corresponds to a facet of minimal volume not only for planar bipartite graphs, but for arbitrary bipartite graphs. Furthermore, it is conjectured that the standard orientation not only minimizes the volume of the corresponding facet of the symmetric edge polytope, but it coefficientwise minimizes the $h^*$-polynomial of the corresponding facet. (The latter would imply the minimization of the volume as well.) 
\end{Rem}

It seems that the structures of the maximizing orientations are much more intricate than that of minimizing orientations. Even in the case of complete graph $K_n$ we do not know the answer if there is no Hadamard tournament on $n$ vertices. In particular, the following problem is open.

\begin{Prob} Which Eulerian tournament $T$ maximizes $\arb(T)$ on $n$ vertices if $n\equiv 1\ (\bmod\ 4)$?
\end{Prob}

The following conjecture seems natural, too.

\begin{Conj} \label{max_orientation_K_n}
Let $n$ be odd. Then the tournament maximizing $\allarb(T_n)$ is Eulerian.
\end{Conj}

Seemingly Conjecture~\ref{max_orientation_K_n} does not follow the short cycle heuristic since the number of directed triangles in a tournament is
$$\binom{n}{3}-\sum_{k=1}^n\binom{d^+_k}{2},$$
and this is maximized for Eulerian orientations. The reason is simple: the short cycle intuition is valid only after fixing the out-degree sequence. Among Eulerian tournaments it is indeed the swirl tournament that maximizes the number of directed $4$-cycles, see \cite{linial2016number}.

The same questions for $K_{n,m}$ are also open.

\begin{Prob} Which Eulerian orientation $O$ maximizes $\arb(O)$ on $K_{n,m}$? Which orientation $O$ maximizes $\allarb(O)$ on $K_{n,m}$?
\end{Prob}

\bibliography{references}

\begin{thebibliography}{10}

\bibitem{aboulker2022decomposing}
Pierre Aboulker, Guillaume Aubian, and Pierre Charbit.
\newblock Decomposing and colouring some locally semicomplete digraphs.
\newblock {\em European Journal of Combinatorics}, 106:103591, 2022.

\bibitem{cayley1849determinants}
Arthur Cayley.
\newblock Sur les d{\'e}terminants gauches.(suite du m{\'e}moire t. xxxii. p.
  119).
\newblock 1849.

\bibitem{ChenDavis}
Tianran Chen and Robert Davis.
\newblock A toric deformation method for solving kuramoto equations.
\newblock {\em Nonlinear Dynamics}, 109, 2022.

\bibitem{creed2010counting}
Patrick~John Creed.
\newblock Counting and sampling problems on eulerian graphs.
\newblock 2010.

\bibitem{DDM19}
Alessio D'Al\`\i, Emanuele Delucchi, and Mateusz Micha\l~ek.
\newblock Many faces of symmetric edge polytopes.
\newblock {\em Electron. J. Combin.}, 29(3):Paper No. 3.24, 42, 2022.

\bibitem{DAli}
Alessio D'Al\`\i, Martina Juhnke-Kubitzke, and Melissa Koch.
\newblock On a generalization of symmetric edge polytopes to regular matroids.
\newblock {\em Int. Math. Res. Not. IMRN}, (14):10844--10864, 2024.

\bibitem{de1951circuits}
Nicolaas~Govert de~Bruijn and Tanja van Aardenne-Ehrenfest.
\newblock Circuits and trees in oriented linear graphs.
\newblock {\em Simon Stevin}, 28:203--217, 1951.

\bibitem{de1992algebraic}
D.~De~Caen, D.~A. Gregory, S.~J. Kirkland, N.~J. Pullman, and J.~S. Maybee.
\newblock Algebraic multiplicity of the eigenvalues of a tournament matrix.
\newblock {\em Linear algebra and its applications}, 169:179--193, 1992.

\bibitem{arithm_symedgepoly}
Akihiro Higashitani, Katharina Jochemko, and Mateusz Micha{\l}ek.
\newblock Arithmetic aspects of symmetric edge polytopes.
\newblock {\em Mathematika}, 65(3):763--784, 2019.

\bibitem{horn2012matrix}
Roger~A Horn and Charles~R Johnson.
\newblock {\em Matrix analysis}.
\newblock Cambridge university press, 2012.

\bibitem{huang1992tournament}
Qing Huang.
\newblock Tournament-like oriented graphs.
\newblock 1992.

\bibitem{isaev2024correlation}
Mikhail Isaev, Brendan~D. McKay, and Rui-Ray Zhang.
\newblock Correlation between residual entropy and spanning tree entropy of
  ice-type models on graphs.
\newblock {\em arXiv preprint arXiv:2409.04989}, 2024.

\bibitem{sym_ribbon}
Tam\'{a}s K\'{a}lm\'{a}n and Lilla T\'othm\'er\'esz.
\newblock Ehrhart theory of symmetric edge polytopes via ribbon structures.
\newblock {\em arXiv:2201.10501}, 2022.

\bibitem{deg_interior_poly}
Tam\'{a}s K\'{a}lm\'{a}n and Lilla T\'othm\'er\'esz.
\newblock Degrees of interior polynomials and parking function enumerators.
\newblock {\em arXiv:2304.03221}, 2023.

\bibitem{ledermann1993note}
Walter Ledermann.
\newblock A note on skew-symmetric determinants.
\newblock {\em Proceedings of the Edinburgh Mathematical Society},
  36(2):335--338, 1993.

\bibitem{linial2016number}
Nati Linial and Avraham Morgenstern.
\newblock On the number of 4-cycles in a tournament.
\newblock {\em Journal of Graph Theory}, 83(3):266--276, 2016.

\bibitem{Sym_edge_appearance}
Tetsushi Matsui, Akihiro Higashitani, Yuuki Nagazawa, Hidefumi Ohsugi, and
  Takayuki Hibi.
\newblock Roots of {E}hrhart polynomials arising from graphs.
\newblock {\em J. Algebraic Combin.}, 34(4):721--749, 2011.

\bibitem{mckay1983spanning}
Brendan~D McKay.
\newblock Spanning trees in regular graphs.
\newblock {\em European Journal of Combinatorics}, 4(2):149--160, 1983.

\bibitem{OT}
Hidefumi Ohsugi and Akiyoshi Tsuchiya.
\newblock The $h^*$-polynomials of locally anti-blocking lattice polytopes and
  their $\gamma$-positivity.
\newblock {\em Discrete and Computational Geometry}, 2020.

\bibitem{reid1972doubly}
KB~Reid and Ezra Brown.
\newblock Doubly regular tournaments are equivalent to skew hadamard matrices.
\newblock {\em Journal of Combinatorial Theory, Series A}, 12(3):332--338,
  1972.

\bibitem{Eulerian_greedoid}
Lilla T\'{o}thm\'{e}r\'{e}sz.
\newblock A geometric proof for the root-independence of the greedoid
  polynomial of {E}ulerian branching greedoids.
\newblock {\em J. Combin. Theory Ser. A}, 206:Paper No. 105891, 21, 2024.

\bibitem{tutte1948dissection}
William~T Tutte.
\newblock The dissection of equilateral triangles into equilateral triangles.
\newblock In {\em Mathematical Proceedings of the Cambridge Philosophical
  Society}, volume~44, pages 463--482. Cambridge University Press, 1948.

\bibitem{van2001course}
Jacobus~Hendricus Van~Lint and Richard~Michael Wilson.
\newblock {\em A course in combinatorics}.
\newblock Cambridge university press, 2001.

\end{thebibliography}
\bibliographystyle{plain}

\end{document}